\documentclass[a4paper,12pt]{article}     
\usepackage{amsmath,amscd,amssymb}        
\usepackage{latexsym}                     
\usepackage[english, german]{babel}                
\usepackage{xypic}

\usepackage{theorem}                 
\usepackage{color}

\newfont{\cyrfnt}{wncyr10}

\newtheorem{theorem}{Theorem}

\newtheorem{proposition}{Proposition}
\newtheorem{lemma}{Lemma}

\theorembodyfont{\rmfamily}
\newtheorem{remark}{Remark}
\newtheorem{example}{Example}

\newenvironment{definition}
{\smallskip\noindent{\bf Definition\/}:}{\smallskip\par}

\newenvironment{examples}
{\smallskip\noindent{\bf Examples\/}.}{\smallskip\par}

\newenvironment{proof}{\begin{ProofwCaption}{Proof}}{\end{ProofwCaption}}
\newenvironment{proof*}[1]{\begin{ProofwCaption}{{#1}}}{\end{ProofwCaption}}
\newenvironment{ProofwCaption}[1]%
  {\addvspace\theorempreskipamount \noindent{\it #1.}\rm}%
  {\qed \par \addvspace\theorempostskipamount}
\newcommand{\qedsymbol}{{\rm $\Box$}}
\newcommand{\qed}{\hfill\qedsymbol}

\newcommand{\CC}{{\mathbb C}}

\newcommand{\PP}{{\mathbb P}}

\newcommand{\ZZ}{{\mathbb Z}}

\newcommand{\Ind}{{\rm Ind}\,}

\title{A version of the Berglund--H\"ubsch--Henningson duality with non-abelian groups}
\author{Wolfgang Ebeling and Sabir M.~Gusein-Zade
\thanks{Partially supported by DFG. The work of the second author
(Sections~\ref{sect:NC-Saito}, \ref{sect:PC}, and \ref{sect:main})
was supported by the grant 16-11-10018 of the Russian Science Foundation.
Keywords: group action, invertible polynomial, mirror symmetry, Berglund--H\"ubsch--Henningson duality,
equivariant Euler characteristic, Saito duality.
Mathematical Subject Classification -- MSC2010: 14J33, 57R18, 32S55, 19A22
}
}
\date{}

\begin{document}
\selectlanguage{english}

\maketitle

\begin{abstract}
A.~Takahashi suggested a conjectural
method to find mirror symmetric pairs consisting of invertible polynomials and  symmetry groups generated by
some diagonal symmetries and some permutations of variables.
Here we generalize the Saito duality between Burnside rings to a case of non-abelian groups and prove a
``non-abelian'' generalization of the statement about the equivariant Saito duality property for invertible polynomials.
It turns out that the statement holds only under a special condition on the action of the subgroup of the
permutation group called here PC (``parity condition''). An inspection of data on Calabi--Yau threefolds
obtained from quotients by non-abelian groups shows that the pairs found
on the basis of the method of Takahashi have symmetric pairs of Hodge numbers if and only if they satisfy PC.
\end{abstract}

\section{Introduction} \label{sect:Intro}
Mirror symmetry in mathematics started from the celebrated observation by physicists that there existed pairs
of Calabi--Yau manifolds with symmetric sets of Hodge numbers.
P.~Berglund, T.~H\"ubsch and M.~Henningson (\cite{BH1}, \cite{BH2}) found a method to construct mirror symmetric
Calabi--Yau manifolds using so-called invertible polynomials: see details below. They considered pairs $(f,G)$
consisting of an invertible polynomial $f$ and a (finite abelian) group $G$ of diagonal symmetries of $f$.
To a pair $(f,G)$ one associates the Berglund--H\"ubsch--Henningson (BHH) dual pair $(\widetilde{f},\widetilde{G})$.
These pairs are also considered as orbifold Landau--Ginzburg models. 
The construction in \cite{BH1}, \cite{BH2} involves the quotient stacks of the hypersurfaces
defined by $f=0$ and $\widetilde{f}=0$ in weighted projective spaces by the groups $G$ and $\widetilde{G}$
respectively.
There were found a number of symmetries
of invariants corresponding to BHH dual pairs. For example, in \cite{Kawai-Yang} and \cite{BH2}, there was
discovered a duality of certain elliptic genera of dual pairs. In \cite{CR-2011}, there were
constructed isomorphisms between certain Chen--Ruan orbifold cohomology groups of dual pairs of Calabi--Yau type.
There were also found symmetries between invariants related to the Milnor fibres of $f$ and $\widetilde{f}$.
In \cite{EG-MMJ}, it was shown that the reduced orbifold Euler
characteristics of the Milnor fibres of dual pairs coincide up to sign. In \cite{EGT}, it was shown that the orbifold
E-functions (generating functions for the exponents of the monodromy actions on orbifold versions of the mixed Hodge
structures on the Milnor fibres) of dual pairs possess a certain symmetry property. In \cite{EG-BLMS}, there was
constructed a duality between the Burnside rings of a finite abelian group and of its group of characters which was
interpreted as the Saito duality in some cases. For the groups of diagonal symmetries of invertible polynomials
this duality is related to the BHH duality. In \cite{EG-BLMS}, it was shown that the reduced equivariant Euler
characteristics of the Milnor fibres of dual invertible polynomials with the actions of their groups of diagonal
symmetries defined as elements of the Burnside rings of the groups are Saito dual to each other
up to sign.

In \cite{Oguiso-Yu} and \cite{Yu}, there were presented many Calabi--Yau threefolds obtained as crepant resolutions
of quotient varieties of smooth quintic threefolds by non-abelian symmetry groups. The data computed for them
indicates that there might be mirror dual pairs. The majority of these threefolds are defined by invertible
polynomials with the symmetry groups generated by subgroups of the groups of diagonal symmetries and subgroups of
the group $S_5$ of permutations of variables.

A.~Takahashi (personal communication) suggested a conjectural method to find mirror symmetric pairs. The method associates to a pair
$(f, \widehat{G})$ with a non-abelian group $\widehat{G}$ of a certain type (namely, a semi-direct product
$G\rtimes S$ of a subgroup $G$ of the group $G_f$ of the diagonal symmetries of $f$ and a subgroup $S\subset S_n$
of the group of permutations of variables) a dual pair $(\widetilde{f}, \widetilde{\widehat{G}})$. An analysis
of the data in \cite{Yu} shows that in some cases the method detects pairs which can be considered as mirror
symmetric ones, whereas in some other cases it fails.

Here we generalize the Saito duality between Burnside
rings to a case of non-abelian groups and prove a ``non-abelian'' generalization of the statement from \cite{EG-BLMS}.
Namely, we compute the equivariant Euler characteristic of the Milnor fibre of an invertible
polynomial with the action of the semidirect product of the group $G_f$ of the diagonal symmetries of $f$ and of a
subgroup $S$ of the group $S_n$ of permutations of variables as an element of a version $A^{\rtimes}(G_f \rtimes S)$
of the Burnside ring. We define a non-abelian version of the Saito duality as a group isomorphism between
$A^{\rtimes}(G_f \rtimes S)$ and $A^{\rtimes}((G_f)^* \rtimes S)$
($(G_f)^*={\rm Hom\,}(G_f,\CC^*)=G_{\widetilde{f}}$) and compare the reduced equivariant Euler characteristics
of the Milnor fibres for Berglund--H\"ubsch--Henningson--Takahashi dual pairs. We show that these reduced
equivariant Euler characteristics are up to sign Saito dual to each other only under a special condition
on the action of the subgroup of the permutation group called here PC (``parity condition'').
We try to analyze a relation of the PC-condition with the mirror symmetry.
An inspection of \cite[Table~5]{Yu} shows that the pairs found on the basis of the method of Takahashi
(there are 13 of them with non-trivial groups $S$) might be considered as mirror symmetric ones (have
symmetric pairs of Hodge numbers) if and only if they satisfy PC. This indicates that the condition PC seems to
be necessary for the mirror symmetry of Berglund--H\"ubsch--Henningson--Takahashi dual pairs.

The authors are grateful to A.~Takahashi for explaining his ideas on non-abelian mirror symmetry
and to the referees of the paper for a number of useful suggestions.

\section{Invertible polynomials and their symmetry groups} \label{sect:invertible_poly}
An invertible polynomial in $n$ variables is a quasihomogeneous polynomial $f$ with the number of monomials 
equal to the number $n$ of variables (that is 
\[ f(x_1, \ldots , x_n)=\sum_{i=1}^n a_i \prod_{j=1}^n x_j^{E_{ij}} \, ,
\]
where $a_i$ are non-zero complex numbers and $E_{ij}$ are non-negative integers) such that the matrix
$E=(E_{ij})$ is non-degenerate and $f$ has an isolated critical point at the origin. 

\begin{remark}
The condition $\det{E}\ne 0$ is equivalent to the condition that the weights $q_1$, \dots, $q_n$ of
the variables in the polynomial $f$ are well defined 
(if one assumes the quasidegree to be equal to $1$). In fact they are defined by the equation
$$
E\cdot(q_1, \ldots, q_n)^T=(1,\ldots, 1)^T\,.
$$
Without loss of generality one may assume that all the coefficients $a_i$ are equal to $1$.
\end{remark}

A classification of invertible polynomials is given in \cite{KS}. Each invertible polynomial is the direct
(``Sebastiani--Thom'') sum of atomic polynomials in different sets of variables of the following types:
\begin{enumerate}
 \item[1)] chains: $x_1^{p_1}x_2+x_2^{p_2}x_3+\ldots+x_{m-1}^{p_{m-1}}x_m+x_m^{p_m}$, 
 $m\ge 1$\,;
 \item[2)] loops: $x_1^{p_1}x_2+x_2^{p_2}x_3+\ldots+x_{m-1}^{p_{m-1}}x_m+x_m^{p_m}x_1$, 
 $m\ge 2$\,.
\end{enumerate}

The group of the diagonal symmetries of $f$ is
$$
G_f=\{\underline{\lambda}=(\lambda_1, \ldots, \lambda_n)\in(\CC^*)^n:
f(\lambda_1x_1, \ldots, \lambda_nx_n)=f(x_1, \ldots, x_n)\}\,,
$$
where $\CC^*=\CC\setminus\{0\}$ is the set of non-zero complex numbers.
One can see that $G_f$ is an abelian group of order $\vert\det E\,\vert$. 
The Milnor fibre $V_f=\{x\in\CC^n:f(x)=1\}$ of the invertible polynomial $f$
is a complex manifold of dimension $n-1$ with the natural action of the group $G_f$. 

The Berglund--H\"ubsch transpose of $f$ is
$$
\widetilde{f}(x_1, \ldots, x_n)=\sum_{i=1}^n \prod_{j=1}^n x_j^{E_{ji}}\,.
$$
The group $G_{\widetilde{f}}$ of the diagonal symmetries of $\widetilde{f}$ is in a canonical way isomorphic
to the group $G_f^*={\rm Hom}(G_f,\CC^*)$ of characters of $G_f$ (see, e.~g., \cite[Proposition 2]{EG-BLMS}).
For a subgroup $G$ of $G_f$, the {\em Berglund--H\"ubsch--Henningson} (BHH) {\em dual} to the pair $(f,G)$
is the pair $(\widetilde{f},\widetilde{G})$, where $\widetilde{G}\subset G_{\widetilde{f}}=G_f^*$ is the subgroup
of characters of $G_f$ vanishing (i.~e.\ being equal to $1$) on the subgroup $G$.

Let the permutation group $S_n$ act on the space $\CC^n$ by permuting the variables. If an invertible polynomial
$f$ is invariant with respect to a subgroup $S\subset S_n$, then it is invariant with respect to the
semi-direct product $G_f\rtimes S$ (defined by the natural action of $S$ on $G_f$). 
Here the operation on $G_f\rtimes S$ is defined as follows. For
$(\underline{\lambda}, \sigma), (\underline{\mu}, \tau) \in G_f\rtimes S$ we define
$$
(\underline{\lambda}, \sigma) \cdot (\underline{\mu}, \tau) :=
(\underline{\lambda} \sigma(\underline{\mu}), \sigma \tau)\,,
$$
where 
$$
\sigma(\underline{\mu}):= (\sigma(\mu_1), \ldots , \sigma(\mu_n)) = (\mu_{\sigma^{-1}(1)}, \ldots , \mu_{\sigma^{-1}(n)})
$$
and the action of the product of $\sigma, \tau \in S$ on $\underline{\mu} \in G_f$ is defined as
$$
(\sigma \tau)(\underline{\mu}) := (\mu_{\tau^{-1}\sigma^{-1}(1)}, \ldots , \mu_{\tau^{-1}\sigma^{-1}(n)}).
$$
The semi-direct product
$G_f\rtimes S$ acts on $\CC^n$ 
in the following way. 
For $(\underline{\lambda}, \sigma) \in G_f\rtimes S$ and $\underline{x}=(x_1, \ldots , x_n)$ we have
$$
(\underline{\lambda}, \sigma)\underline{x} = (\lambda_1x_{\sigma^{-1}(1)}, \ldots , \lambda_nx_{\sigma^{-1}(n)}).
$$
The action of the group $G_f\rtimes S$ 
can be restricted
to the Milnor fibre
$V_f=\{\underline{x}\in\CC^n: f(\underline{x})=1\}$ of the polynomial $f$.

Let $G$ be a subgroup of $G_f$ invariant with respect to the group $S$, i.~e., $\sigma G=G$
for any $\sigma \in S$. In this case the semidirect product
$G\rtimes S$ is defined and the BHH dual subgroup $\widetilde G$ is also invariant with respect to $S$.

\begin{definition}
 The {\em Berglund--H\"ubsch--Henningson--Takahashi} (BHHT) {\em dual} to the pair $(f,G\rtimes S)$ is the pair
$(\widetilde{f},\widetilde{G}\rtimes S)$.
\end{definition}

The Burnside ring $A(H)$ of a finite group $H$ is the Grothendieck ring of finite $H$-sets.
This means that $A(H)$ is the abelian group generated by the classes $[(Z,H)]$
of finite $H$-sets modulo the relations
\begin{enumerate}
 \item[1)] if $(Z,H)$ is isomorphic to $(Z',H)$, i.~e., if there exists a bijection between $Z$ and $Z'$ preserving
 the $H$-action, then $[(Z,H)]=[(Z',H)]$;
 \item[2)] if $[(Z_1,H)]$ and $[(Z_2,H)]$ are finite $H$-sets, then $[(Z_1 \sqcup Z_2,H)]=[(Z_1,H)]+[(Z_2,H)]$.
 \end{enumerate}
 The multiplication in $A(H)$ is defined by the Cartesian product of the sets with the natural (diagonal)
 $H$-action.
As an abelian group, $A(H)$ is freely generated by the classes $[H/K]$ of the quotient sets $H/K$ for representatives
$K$ of the conjugacy classes $[K]$ of subgroups of $H$. For an $H$-space $X$ and for a point $x\in X$ the isotropy
subgroup of $x$ is $H_x :=\{g\in H: gx=x\}$. For a subgroup $K\subset H$ the set of fixed points of $K$ (that is
points $x$ with $H_x\supset K$) is denoted by $X^K$; the set of points $x\in X$ with the isotropy subgroup $K$
is denoted by $X^{(K)}$, the set of points $x\in X$ with the isotropy subgroup conjugate to $K$ is denoted by
$X^{([K])}$. The  {\em equivariant Euler characteristic} of a topological $H$-space $X$ 
was introduced in \cite{TtD} as
the element of the Burnside ring $A(H)$ defined by
\begin{equation}\label{eq:equiv_chi}
 \chi^H(X) :=\sum_{[K]\in{\rm Conjsub\,}H}\chi(X^{([K])}/H)[H/K]\,,
\end{equation}
where ${\rm Conjsub\,}H$ is the set of conjugacy classes of subgroups of $H$.
The  {\em reduced equivariant Euler characteristic}  $\overline{\chi}^H(X)$ is
$\chi^H(X)-\chi^H(pt)=\chi^H(X)-[H/H]$. The equivariant Euler characteristic is a universal
additive topological invariant of varieties or manifolds with $H$-actions: see, e.~g., \cite{GZ-RMS}.
Therefore it determines other additive topological invariants, in particular, the standard Euler characteristic,
the Euler characteristic of the quotient and the orbifold Euler characteristic.
It does not determine analytic or algebraic invariants, say, (orbifold) Hodge numbers.

If $H$ is a subgroup of a finite group $G$, one has the {\em reduction} and the {\em induction} operations
${\rm Red}^G_H$ and ${\rm Ind}^G_H$ which convert $G$-spaces to $H$-spaces and $H$-spaces to $G$-spaces
respectively. The reduction ${\rm Red}^G_HX$ of a $G$-space $X$ is the same space considered with the action
of the
subgroup. The induction ${\rm Ind}^G_HX$ of an $H$-space $X$ is the quotient space $(G\times X)/\sim$,
where the equivalence relation $\sim$ is defined by: $(g_1,x_1)\sim(g_2,x_2)$ if (and only if) there exists $h\in H$
such that $g_2=g_1h$, $x_2=h^{-1}x_1$; the $G$-action on it is defined in the natural way. Applying the reduction and
the induction operations to finite $G$- and $H$-sets respectively, one gets the {\em reduction homomorphism}
${\rm Red}^G_H:A(G)\to A(H)$ and the {\em induction homomorphism} ${\rm Ind}^G_H:A(H)\to A(G)$. For a subgroup $K$
of $H$, one has ${\rm Ind}^G_H[H/K]=[G/K]$. The reduction homomorphism is a ring homomorphism, whereas the
induction one is a homomorphism of abelian groups.

For a finite abelian group $H$, let $H^*={\rm Hom}(H,\CC^*)$ be its group of characters. Just as for a subgroup
of $G_f$ above, for a subgroup $K\subset H$, the (BHH) dual subgroup of $H^*$ is
$$
\widetilde{K} :=\{\alpha\in H^*: \alpha(g)=1 \text{ for all } g\in K\}\,.
$$
The equivariant Saito duality (see \cite{EG-BLMS}) is the group homomorphism $D_H:A(H)\to A(H^*)$ defined by
$D_H([H/K]) :=[H^*/\widetilde{K}]$. In \cite{EG-BLMS}, it was shown that the reduced equivariant Euler characteristics
of the Milnor fibres of Berglund--H\"ubsch dual invertible polynomials $f$ and $\widetilde{f}$ with the actions
of the groups $G_f$ and $G_{\widetilde{f}}$ respectively are Saito dual to each other up to the sign $(-1)^n$.

\section{Non-abelian equivariant Saito duality} \label{sect:NC-Saito}
Let $G$ be a finite abelian group and let $S$ be a finite group with a homomorphism $\varphi:S\to {\text{Aut\,}}G$.
These data determine the semi-direct product $\widehat{G}=G\rtimes S$. Let $A^{\rtimes}(G\rtimes S)$ be the
Grothendieck group of finite $\widehat{G}$-sets with the isotropy subgroups of points conjugate to
$H\rtimes T\subset G\rtimes S$, where $H$ and $T$ are subgroups of $G$ and of $S$ respectively such that,
for $\sigma\in T$, the automorphism $\varphi(\sigma)$ preserves $H$. (The semidirect product structure on $H\rtimes T$
is defined by the homomorphism $\varphi_{\vert T}:T\to {\text{Aut\,}}H$.) The group $A^{\rtimes}(G\rtimes S)$ is
a subgroup of the Burnside ring $A(\widehat{G})$ of the group $\widehat{G}$. 
It is the free abelian group generated by 
the elements $\left[G\rtimes S/H\rtimes T\right]$ corresponding to
the conjugacy classes of the subgroups of the form $H\rtimes T$.
An element of $A^{\rtimes}(G\rtimes S)$ can be written in a unique way as 
$$
\sum_{[H\rtimes T]\in {\text{Conjsub\, }}\widehat{G}} a_{H\rtimes T}\left[G\rtimes S/H\rtimes T\right]
$$
with integers $a_{H\rtimes T}$.

\begin{proposition}\label{prop:conj}
 Subgroups $H_1\rtimes T_1$ and $H_2\rtimes T_2$ are conjugate in $G\rtimes S$ if and only if they are conjugate by
 an element $\sigma\in S\subset \widehat{G}$, i.~e.\ $T_2=\sigma^{-1}T_1\sigma$, $H_2=\varphi(\sigma)H_1$.
\end{proposition}

\begin{proof}
 Each element of $G\rtimes S$ is the  pair of an element of $G$ and an element of $S$. 
 To prove the statement, we have to show that, if the conjugation by an element of $G\subset G\rtimes S$ sends
 a subgroup $H\rtimes T$ to a subgroup of the form $K\rtimes T$, then $K=H$. We have
 $$
 g^{-1}(h,\sigma)g=\left(\left(g^{-1}\varphi(\sigma(g)\right)h,\sigma\right)\,.
 $$
 The set $\left(g^{-1}\varphi(\sigma(g)\right)H$ is a coset of the subgroup $H$. Thus it is a subgroup $K$ of $G$
 if and only if it coincides with $H$.
\end{proof}

Let $G^*={\text{Hom\, }}(G,\CC^*)$ be the group of characters on $G$. One has $G^{**}\cong G$ (canonically).
The homomorphism $\varphi:S\to {\text{Aut\,}}G$ induces a natural homomorphism $\varphi^*:S\to {\text{Aut\,}}G^*$:
$\langle\varphi^*(\sigma)\alpha, g\rangle=\langle\alpha, \varphi(\sigma^{-1})g\rangle$, where
$\langle\alpha, g\rangle:=\alpha(g)$. Let $\widehat{G}^*:=G^*\rtimes S$ be the semidirect product defined by the
homomorphism $\varphi^*$.
One can see that, if $\varphi(\sigma)$ preserves a subgroup $H\subset G$, then $\varphi^*(\sigma)$
preserves the subgroup $\widetilde{H}\subset G^*$. Thus for a semidirect product $H\rtimes T\subset G\rtimes S$
one has the semidirect product $\widetilde{H}\rtimes T\subset G^*\rtimes S$.

\begin{proposition}\label{prop:conj_dual}
 Subgroups $H_1\rtimes T_1$ and $H_2\rtimes T_2$ are conjugate in $G\rtimes S$ if and only if the subgroups
 $\widetilde{H_1}\rtimes T_1$ and $\widetilde{H_2}\rtimes T_2$ are conjugate in $G^*\rtimes S$.
\end{proposition}

\begin{proof}
 This is a direct consequence of Proposition~\ref{prop:conj}: if the subgroups $H_1\rtimes T_1$ and $H_2\rtimes T_2$
 are conjugate by an element $\sigma\in S$, then the subgroups $\widetilde{H_1}\rtimes T_1$ and
 $\widetilde{H_2}\rtimes T_2$ are conjugate by this element as well.
\end{proof}

\begin{definition}
 The {\em (``non-abelian'') equivariant Saito duality} corresponding to the group $\widehat{G}=G\rtimes S$
 is the group homomorphism $D^{\rtimes}_{\widehat{G}}:A^{\rtimes}(G\rtimes S)\to A^{\rtimes}(G^*\rtimes S)$
 defined (on the generators) by
 $$
 D^{\rtimes}_{\widehat{G}}([G\rtimes S/H\rtimes T])=
 [G^*\rtimes S/\widetilde{H}\rtimes T]\,.
 $$
\end{definition}

By Proposition~\ref{prop:conj_dual}, the homomorphism $D^{\rtimes}_{\widehat{G}}$ is well-defined.
One can see that $D^{\rtimes}_{\widehat{G}}$ is an isomorphism of the groups $A^{\rtimes}(G\rtimes S)$ and
$A^{\rtimes}(G^*\rtimes S)$ and $D^{\rtimes}_{\widehat{G^*}}D^{\rtimes}_{\widehat{G}}={\rm id}$
($\widehat{G^*}=G^*\rtimes S$)\,.

For a subgroup $S'\subset S$ one has the natural homomorphism
$\Ind_{G\rtimes S'}^{G\rtimes S}:A^{\rtimes}(G\rtimes S')\to A^{\rtimes}(G\rtimes S)$ sending 
the generator $\left[G\rtimes S'/H\rtimes T\right]$ to the generator $\left[G\rtimes S/H\rtimes T\right]$.
This homomorphism commutes with the Saito duality, i.~e\ the diagram
$$
\begin{CD}
A^{\rtimes}(G\rtimes S') @>D^{\rtimes}_{G\rtimes S'}>> A^{\rtimes}(G^*\rtimes S')\\
@VV\Ind_{G\rtimes S'}^{G\rtimes S}V
@VV\Ind_{G^*\rtimes S'}^{G^*\rtimes S}V
\\
A^{\rtimes}(G\rtimes S) @>D^{\rtimes}_{G\rtimes S}>> A^{\rtimes}(G^*\rtimes S)
\end{CD}
$$
is commutative.

\section{Equivariant Euler characteristic of the Milnor fibre} \label{sect:Euler_Milnor}
Let $f$ be an invertible polynomial in $n$ variables, let $G_f$ be the group of the diagonal symmetries of $f$,
and let $S$ be a subgroup of $S_n$ preserving $f$.

An atomic polynomial of chain type is not preserved by a permutation of variables different
from the identical one. An atomic polynomial of loop type may have non-trivial symmetries by permutations of
variables. This may happen in two ways. First, a loop 
\begin{equation} \label{eq:loop}
x_1^{p_1}x_2+x_2^{p_2}x_3+\ldots+x_{m-1}^{p_{m-1}}x_m+x_m^{p_m}x_1
\end{equation}
can be such that its length $m$ is divisible by an integer $\ell$, $\ell<m$, (that is $m=k\ell$ with $k>1$) and the
sequence $p_1, \ldots , p_m$ of its exponents is $\ell$-periodic, i.~e.\ $p_{i+\ell}=p_i$ for $i=1, \ldots , m$.
Here and below the indices ($i$ and $i+\ell$ in this case) are taken modulo $m$. A symmetry sends the variable
$x_i$ to the variable $x_{i+s\ell}$ for some $s=1,\ldots, k-1$. Symmetries of this sort will be called
{\em rotations}. Besides that one may have a flip of a loop (\ref{eq:loop}) which sends the variable $x_i$ to
the variable $x_{r-i}$ for some $r=1, 2, \ldots , m$ (again
the indices are considered modulo $m$). A permutation of variables of this sort preserves the
loop if and only if all the exponents $p_i$ are equal to 1. In this case either (\ref{eq:loop})
is not an invertible polynomial (since it has a non-isolated critical point at the origin) or it has
a non-degenerate critical point (i.~e.\ of type $A_1$) at the origin. Therefore we will exclude this type of
symmetries of loops from consideration.

The invertible polynomial $f$ is the direct (Sebastiani--Thom) sum of atomic polynomials (blocks).
In these terms the action of its symmetry group $S\subset S_n$ can be described in the following way. There are
some blocks which are permuted by the group $S$ (all permuted blocks are isomorphic). This means that, if a block
goes to itself under the action of an element $\sigma\in S$, then $\sigma$ acts on the block identically.
We shall say that these blocks are of the first type. All chain blocks are of the first type. Other blocks will be
said to be of the second type. All of them are loops and if $\sigma\in S$ sends a block to itself, then it acts
on the block (a loop) by a rotation.

One can see that the dual polynomial $\widetilde{f}$ is invariant with respect to the group $S$ as well.

To have the possibility to discuss a non-abelian equivariant Saito duality for dual invertible polynomials,
one has to know that the equivariant Euler characteristic of the Milnor fibre $V_f$ of a polynomial $f$ with
the $G_f \rtimes S$-action belongs to the subgroup $A^{\rtimes}(G_f \rtimes S) \subset A(G_f \rtimes S)$.
This will be proved in Section~\ref{sect:main} with the use of the following statement.

\begin{sloppypar}

\begin{proposition} \label{prop:exT}
If, for a conjugacy class $[ \check{T} ] \in {\rm Conjsub\,} (G_f \rtimes S)$, one has
$((\CC^\ast)^n)^{([ \check{T} ])} \neq \emptyset$, then the class $[ \check{T} ]$ contains
a representative of the form $\{ e \}Ê\times T$. 
Moreover, subgroups $\{ e \}Ê\times T_1$ and $\{ e \}Ê\times T_2$ are conjugate in $G_f \rtimes S$ if and only if
$T_1$ and $T_2$ are conjugate in $S$.
\end{proposition}

\end{sloppypar}

\begin{proof}
Let $T=\pi(\check{T})$ where $\check{T}$ is a representative of the conjugacy class $[\check{T}]$ and 
$\pi: G_f \rtimes S \to S$ is the quotient map. Since the action of $G_f$ on $(\CC^\ast)^n$ is free 
and $((\CC^\ast)^n)^{(\check{T})} \neq \emptyset$,
for any $\sigma \in T$ there exists a unique $\underline{\lambda}= \underline{\lambda}(\sigma) \in G_f$
such that $(\underline{\lambda}(\sigma), \sigma) \in \check{T}$.
(If $\underline{\lambda}'\underline{x}=\underline{\lambda}''\underline{x}=\underline{x}$
for $\underline{x}\in \left(\CC^*\right)^n$, then 
$(\lambda_1'x_{\sigma^{-1}(1)}, \ldots, \lambda_n'x_{\sigma^{-1}(n)})=
(\lambda_1''x_{\sigma^{-1}(1)}, \ldots, \lambda_n''x_{\sigma^{-1}(n)})$,
what is impossible for $\underline{\lambda}'\ne\underline{\lambda}''$
since the action of $G_f$ on $\left(\CC^*\right)^n$ is free.)
One has $(\underline{\lambda}(\delta), \delta)\cdot(\underline{\lambda}(\sigma), \sigma)
=(\underline{\lambda}(\delta)\cdot \delta(\underline{\lambda}(\sigma)), \delta\sigma)$
and therefore
\begin{equation}\label{eqn:lambda_product}
 \underline{\lambda}(\delta\sigma)=\underline{\lambda}(\delta)\delta(\underline{\lambda}(\sigma))\,.
\end{equation}
Let $\underline{x} \in ((\CC^\ast)^n)^{(\check{T})}$,
i.~e.\ $(\underline{\lambda}(\sigma), \sigma) \underline{x} = \underline{x}$ for all $\sigma \in T$.
We shall show that there exists a $\underline{\mu} \in G_f$ such that
\begin{equation} \label{cond:mu}
\underline{\mu} \cdot \underline{y} \in ((\CC^\ast)^n)^{\check{T}} \mbox{ for all }
\underline{y}=(y_1, \ldots , y_n) \in ((\CC^\ast)^n)^T\, .
\end{equation} 
Here $\underline{y} \in ((\CC^\ast)^n)^T$ means that $y_{\sigma^{-1}(i)} = y_i$, $i=1, \ldots, n$,
for all $\sigma \in T$.

The condition $\underline{\mu} \cdot \underline{y} \in ((\CC^\ast)^n)^{\check{T}}$,
i.~e.\ $(\underline{\lambda}(\sigma), \sigma)\underline{\mu} \cdot \underline{y}=
\underline{\mu} \cdot \underline{y}$, means that
\begin{equation} \label{eqn:mu}
\underline{\lambda}(\sigma)\sigma(\underline{\mu})=\underline{\mu}\,.
\end{equation} 
This will imply that $((\CC^\ast)^n)^{(\underline{\mu}^{-1} \check{T} \underline{\mu})}$ is contained in
$((\CC^\ast)^n)^{\{ e \}Ê\times T}$ and, in particular,
$\underline{\mu} \cdot \underline{x} \in ((\CC^\ast)^n)^{(\{ e \}Ê\times T)}$,
i.~e.\ $\{ e \}Ê\times T \in [\check{T}]$.

It is sufficient to prove the existence of $\underline{\mu}$ for a set of blocks from
the same $T$-orbit. The action of the group $T$ on blocks of the polynomial $f$ was described above. 
Each element of $T$ sends a block of any type to an isomorphic block. Moreover, if a block is of loop type
with a rotation (i.~e.\ there are some elements of $T$ which rotate the block), then any element of $T$ sends
it to an isomorphic block with a rotation. According to the description above,
an orbit of $T$ consists either of blocks which are permuted (blocks of the first type)
or of blocks (of loop type) which are permuted with a rotation (blocks of the second type).
Let us consider these two cases.

Case 1. Suppose that
the orbit consists of blocks $B_\alpha$, $\alpha \in A$, of the first type. Here $A$ is
a certain indexing set. The group $G_{\boxplus B_{\alpha}}$ of the diagonal symmetries of the direct sum of
these blocks is the direct sum of the groups $G_{B_\alpha}$ of symmetries of the block $B_\alpha$ (all of
them are isomorphic).
Let us denote an element $\underline{\lambda} \in G_{\boxplus B_{\alpha}}$ by $\{ \underline{\lambda}_\alpha \}$
with $\underline{\lambda}_\alpha \in G_{B_\alpha}$. 
Let $(\CC^\ast)_{\boxplus B_{\alpha}}$ be the torus corresponding to the variables of the blocks
$B_\alpha$, $\alpha \in A$.
We shall also write points $\underline{y} \in (\CC^\ast)_{\boxplus B_{\alpha}}$ as
$\{ \underline{y}_\alpha \}$. 
Let us fix $\alpha_0 \in A$. A point $\underline{y} \in (\CC^\ast)_{\boxplus B_{\alpha}}$
invariant with respect to the $T$-action is determined by
its part $\underline{y}_{\alpha_0}$. The condition (\ref{eqn:mu}) can be written as
\begin{equation}\label{eqn:arbitrary_alpha}
\underline{\lambda}_{\alpha}(\sigma)\underline{\mu}_{\sigma^{-1}(\alpha)}=\underline{\mu}_{\alpha}
\end{equation}
for all $\alpha$. For $\alpha=\alpha_0$ we have 
\begin{equation}\label{eqn:alpha_0}
\underline{\lambda}_{\alpha_0}(\sigma)\underline{\mu}_{\sigma^{-1}(\alpha_0)}=\underline{\mu}_{\alpha_0}\,.
\end{equation}
Thus we can take $\underline{\mu}_{\alpha_0}= \underline{1}$,
\begin{equation}\label{eqn:def_mu_1}
 \underline{\mu}_{\sigma^{-1}(\alpha_0)} = \underline{\lambda}_{\alpha_0}(\sigma)^{-1}.
\end{equation}
We have to check (\ref{eqn:arbitrary_alpha}) for arbitrary $\alpha=\delta(\alpha_0)$.
Substituting $\underline{\mu}$ from (\ref{eqn:def_mu_1}) we get
$\underline{\lambda}_{\delta(\alpha_0)}(\sigma)\cdot\left(\underline{\lambda}_{\alpha_0}(\delta\sigma)\right)^{-1}=
\left(\underline{\lambda}_{\alpha_0}(\delta)\right)^{-1}$,
i.~e.\ $\underline{\lambda}_{\alpha_0}(\delta\sigma)=
\underline{\lambda}_{\alpha_0}(\delta)\cdot\underline{\lambda}_{\delta(\alpha_0)}(\sigma)\cdot$.
This is a particular case of (\ref{eqn:lambda_product}).

Case 2. Suppose that
the orbit consists 
of blocks $B_\alpha$, $\alpha \in A$, of the second type (all of them are loops).
The coordinates in the block $B_{\alpha}$
will be denoted by $x_{\alpha,i}$. For a fixed $\alpha_0 \in A$, the subgroup $T_{\alpha_0}$ of $T$ sending
the block $B_{\alpha_0}$ to itself acts on this block in the following way. We shall omit the index $\alpha_0$,
i.~e.\ we shall denote $x_{\alpha_0,i}$ by $x_i$ for short. The block $B_{\alpha_0}$ is of the form
\begin{equation}
\sum_{i=1}^{k\ell} x_i^{p_i}x_{i+1},
\end{equation}
where the index $i$ is considered modulo $k\ell$ and $p_i=p_{i+\ell}$
for all $i=1, \ldots , k\ell$. An orbit of the action of the group
$T_{\alpha_0}$ on the variables $x_1, \ldots , x_{k\ell}$ consists of all $x_i$ with $i$ congruent to each other
modulo $\ell$. An element $\underline{y}$ invariant with respect to the action of $T_{\alpha_0}$ is determined
by its coordinates $y_1, \ldots , y_{\ell}$. Let us first define the components $\mu_i$ of $\underline{\mu}$
for $i=1 , \ldots , k\ell$. 
Let $\sigma \in T_{\alpha_0}$ be the rotation $i \mapsto i+\ell$.
Let $\lambda_i[s]:=\lambda_i(\sigma^s)$ for $i=1, \ldots, , k\ell$ and $s=0,1, \ldots , k-1$.
A point $(y_1, \ldots , y_{k\ell})$
is fixed with respect to $\check{T} \cap \pi^{-1}(T_{\alpha_0})$ if 
\begin{equation} \label{eq:**}
\lambda_i[s] y_{i-s\ell} = y_i \quad \mbox{for all integers }i \mbox{ taken modulo }\ell.
\end{equation}
Let $\lambda_i:=\lambda_i[1]$, $P:=p_1 \cdots p_\ell$. The fact that $\underline{\lambda} \in G_{B_{\alpha_0}}$
is equivalent to $\lambda_i^{p_i} \lambda_{i+1}=1$ for all $i$ taken modulo $\ell$.
The condition (\ref{cond:mu})  on $\underline{\mu}$ is 
$\lambda_i[s]\mu_{i-s\ell} y_i=\mu_iy_i$ for $i=1,\ldots, k\ell$, $s=1, \ldots, k$.
It is sufficient to have $\lambda_i[1]\mu_{i-\ell}=\mu_i$ for $i=1,\ldots, k\ell$.
This follows from the equation
\begin{equation}
\lambda_i[s]=\lambda_i[1]\lambda_{i-\ell}[1] \cdots \lambda_{i-(s-1)\ell}[1].
\end{equation}

The condition that $\underline{\mu} \in G_{B_{\alpha_0}}$ means that
\begin{equation} \label{eq:*}
\mu_1^{p_1}\mu_2 = \mu_2^{p_2}\mu_3 = \cdots = \mu_\ell^{p_\ell}\mu_{\ell+1} = \mu_{\ell+1}^{p_1}\mu_{\ell+2} =
\cdots  = \mu_{k\ell}^{p_\ell}\mu_1=1.
\end{equation}
In particular, $\mu_{\ell+1} = \mu_1^{(-1)^\ell P}$. On the other hand,
$\lambda_{\ell+1} \mu_1 =  \mu_{\ell +1}$. Thus 
\begin{equation} \label{eq:****}
\mu_1^{(-1)^\ell P-1} = \lambda_{\ell+1}.
\end{equation}
Therefore $\mu_1$ should be a root of degree $((-1)^\ell P -1)$ of $\lambda_{\ell+1}$. Let us define $\mu_1$ as
any fixed root of this type. The equations (\ref{eq:*}) ($\mu_i^{p_i} \mu_{i+1}=1$) for $i=1, \ldots , k\ell-1$
determine all the components  $\mu_i$ of $\underline{\mu}$ as powers of $\mu_1$. In particular, they give
\begin{equation}
\mu_{k\ell}^{p_\ell} = \mu_1^{(-1)^{k\ell-1} P^k}.
\end{equation}
One has to verify the last equation
\begin{equation} \label{eq:***}
\mu_1^{1-(-1)^{k\ell} P^k} = 1.
\end{equation}
The conditions (\ref{eq:**}) imply that
\begin{equation}
\lambda_{\ell+1} \cdots \lambda_{(k-1)\ell+1}\lambda_1=1, \mbox{ i.~e.\ } \lambda_{\ell+1} =
(\lambda_{2\ell+1} \cdots \lambda_{(k-1)\ell+1}\lambda_1)^{-1}.
\end{equation}
The conditions $\lambda_i^{p_i}\lambda_{i+1} =1$ ($\underline{\lambda} \in G_{B_{\alpha_0}}$) imply that
\begin{eqnarray*}
\lambda_{2\ell+1} & = & \lambda_{\ell+1}^{(-1)^\ell P}, \\
\lambda_{3\ell+1} & = & \lambda_{\ell+1}^{(-1)^{2\ell} P^2}, \\
\vdots & \vdots & \vdots \\
\lambda_{(k-1)\ell+1} & = & \lambda_{\ell+1}^{(-1)^{(k-2)\ell} P^{k-2}}, \\
\lambda_1 & = & \lambda_{\ell+1}^{(-1)^{(k-1)\ell} P^{k-1}}.
\end{eqnarray*}
Therefore we have
\begin{equation}
\lambda_{\ell+1}^{1+(-1)^\ell P + (-1)^{2\ell} P^2 + \cdots + (-1)^{(k-2)\ell} P^{k-2} + (-1)^{(k-1)\ell} P^{k-1}}=1.
\end{equation} 
Together with (\ref{eq:****}) this gives (\ref{eq:***}).

Just as in Case~1, we must have (\ref{eqn:arbitrary_alpha}) and thus (\ref{eqn:alpha_0}).
This means that for an arbtrary $\alpha$ we shall define $\underline{\mu}$ by
$\underline{\mu}_{\sigma^{-1}(\alpha_0)} =
\underline{\lambda}_{\alpha_0}(\sigma)^{-1}\underline{\mu}_{\alpha_0}$.
The condition (\ref{eqn:arbitrary_alpha}) for an arbitrary $\alpha$ is obtained literally as in Case~1.
\end{proof}

\section{Condition PC} \label{sect:PC}
Let $f$ be an invertible polynomial invariant with respect to a subgroup $S \subset S_n$. 
In Section~\ref{sect:main}, it will be shown that
$\overline{\chi}^{\widehat{G}_f}(V_{f})\in A^\rtimes(G_f\rtimes S)$ (and similarly
$\overline{\chi}^{\widehat{G}_{\widetilde{f}}}(V_{\widetilde{f}}) \in A^\rtimes(G_{\widetilde{f}} \rtimes S)$).
If $S=\{ e \}$, one has the equivariant Saito duality
\begin{equation}          
  \overline{\chi}^{\widehat{G}_f}(V_{f})
  =(-1)^n D_{\widehat{G}_{\widetilde{f}}}^{\rtimes} \overline{\chi}^{\widehat{G}_{\widetilde{f}}}(V_{\widetilde{f}})\,. \label{eq:Saito}
\end{equation}
Let us show that Equation~(\ref{eq:Saito}) does not hold in general for $S \neq \{ e \}$.

\begin{example} Let 
\[
f(\underline{x})=\widetilde{f}(\underline{x})=x_1^m+x_2^m+x_3^m+x_4^m,
\quad S= \langle(12)(34), (13)(24) \rangle \cong \ZZ_2 \times \ZZ_2.
\]
The group $G_f (\cong G_{\widetilde{f}})$ is isomorphic to $\ZZ_m^4$. The isotropy subgroup of a point
with respect to the
$G_f \rtimes S$-action on the Milnor fibre $V_f$ is conjugate to $\ZZ_m^r \rtimes T$, where $T$ is one
of the subgroups $S$, $\{ e \}$, $\ZZ_2'$, $\ZZ_2''$, $\ZZ_2'''$ of the group $S$ ($\ZZ_2'$, $\ZZ_2''$, and
$\ZZ_2'''$ are the cyclic subgroups of $S$ of order 2), $0 \leq r \leq 3$. 
Here $r < 4$, since the action of $\ZZ_m^4$ on 
$V_f$ is effective.
We shall compute the coefficients
of $[G_f \rtimes S/\ZZ_m^r \rtimes T]$ in $\overline{\chi}^{\widehat{G}_f}(V_{f})$ with $r=0$. Only these
summands can be Saito dual to the summand $-[G_{\widetilde{f}} \rtimes S/
G_{\widetilde{f}} \rtimes S ]$ in $\overline{\chi}^{\widehat{G}_{\widetilde{f}}}(V_{\widetilde{f}})$.
The coefficient of $[G_f \rtimes S/\{ e \}\rtimes T]$ in $\overline{\chi}^{\widehat{G}_f}(V_{f})$
is equal to the Euler characteristic of the 
quotient space $V_f^{([\{e\} \rtimes T])}/G_f \rtimes S$ (see Equation~(\ref{eq:equiv_chi}) for the equivariant
Euler characteristic).

The isotropy subgroups conjugate to $\{ e\} \rtimes T$ can be met only for points from $V_f \cap (\CC^\ast)^4$.
The points with the isotropy subgroup $\{ e \} \rtimes S$ are of the form $(x,x,x,x)$. There are $m$ of them.
The points with the isotropy subgroup $\check{S}$ conjugate to $\{ e \} \rtimes S$ are obtained
from these ones by the action of the group $G_f \cong \ZZ_m^4$
since $\underline{\mu}^{-1} \check{S} \underline{\mu} = \{ e \} \rtimes S$ for some
$\underline{\mu} \in G_f$. 
Thus there are $m^4$ of them
($m$ elements of $G_f$ permute the $m$ points described above). Therefore the coefficient of
$[G_f \rtimes S/\{e\} \rtimes S]$ in $\overline{\chi}^{\widehat{G}_f}(V_{f})$ is equal to 1.

The points with the isotropy subgroup $\{ e \} \rtimes \ZZ_2'$ with $\ZZ_2'=\langle (12)(34) \rangle$ are
of the form $(x,x,y,y)$ with $x \neq y$. The Euler characteristic of the space
$\{ \underline{x} \in V_f \cap (\CC^\ast)^4 \, : \, x_1=x_2, x_3=x_4 \}$ is equal to $-m^2$. Therefore the Euler
characteristic of the union of the shifts of this set by the elements of the group $G_f$ is equal to $-m^4$
($m^2$ elements of $G_f$ leave this set invariant). The Euler characteristic of $V_f^{([\{ e\} \rtimes \ZZ_2'])}$
is obtained from this one by subtracting the Euler characteristic of $V_f^{([\{e\} \rtimes S])}$ (equal to $m^4$)
and thus is equal to $-2m^4$. Therefore the coefficient of $[G_f \rtimes S/\{ e \} \rtimes \ZZ_2']$ in
$\overline{\chi}^{\widehat{G}_f}(V_{f})$ is equal to $-1$. The same holds for the coefficients of
$[G_f \rtimes S/\{ e \} \rtimes \ZZ_2'']$ and $[G_f \rtimes S/\{ e \} \rtimes \ZZ_2''']$.

In the same way one can show that the coefficient of $[G_f \rtimes S/\{ e \} \rtimes \{ e \}]$ is equal to 1.
Thus all these terms sum up to 
\begin{eqnarray*}
\lefteqn{[G_f \rtimes S/\{ e \} \rtimes S] - [G_f \rtimes S/\{ e \} \rtimes \ZZ_2'] -
[G_f \rtimes S/\{ e \} \rtimes \ZZ_2'']}\\
& &  - [G_f \rtimes S/\{ e \} \rtimes \ZZ_2'''] +[G_f \rtimes S/\{ e \} \rtimes \{ e \}].
\end{eqnarray*}
The first term in this expression is Saito dual to $[G_{\widetilde{f}} \rtimes S/G_{\widetilde{f}} \rtimes S]$
(though with the non-expected sign),
whereas the other terms do not have dual counterparts. The reason for these terms not to vanish is the following one.
For an invertible polynomial $f$ in $n$ variables with the symmetry group $G_f \rtimes S$, the sign of the Euler
characteristic of the set $(V_f \cap (\CC^\ast)^n)^T$ ($T \subset S$) is $(-1)^{\dim (\CC^n)^T-1}$.
This follows, e.~g., from the formula of \cite[Theorem (7.1)]{Varch}
(see Equation~(\ref{eq:Varchenko}) below).
\end{example}

In the example above $\dim (\CC^n)^S =1$ is odd while $\dim (\CC^n)^{\ZZ_2'}=2$ 
and $\dim (\CC^n)^{\{e\}}=0$ are even. In a vague way, one can say that
if the latter dimension would be odd as well, the Euler characteristic of $V_f^{([\{ e\} \rtimes \ZZ_2'])}$
would vanish (being equal to $m^4-m^4$). This gives a hint that the Saito duality for dual invertible polynomials
can only hold if the dimensions of the subspaces $(\CC^n)^T$ for all subgroups $T \subset S$ have the same parity.

\begin{remark}
We might present a simpler example with $f(\underline{x})=x_1^m+x_2^m$, $S=\langle (12) \rangle$, however, in this
case the group $S$ is not contained in ${\rm SL}(2, \CC)$ what is the usual condition in mirror symmetry.
\end{remark}


The example above and the discussion of it motivate the following condition on the action of the group $S$.

\begin{definition}
 We say that a subgroup $S\subset S_n$ satisfies the {\em parity condition} (``PC'' for short) if for each subgroup
 $T\subset S$ one has
 $$
 \dim (\CC^n)^T\equiv n\ \ \mod 2\,,
 $$
 where $(\CC^n)^T=\{\underline{x}\in \CC^n: \sigma \underline{x}= \underline{x} \text{ for all } \sigma\in T\}$.
\end{definition}

\begin{proposition}\label{prop:A_n}
 A subgroup $S\subset S_n$ satisfying PC is contained in the alternating group $A_n\subset S_n$. 
\end{proposition}

\begin{proof}
 Let $\sigma$ be an element of $S$. The dimension of the subspace $(\CC^n)^{\langle\sigma\rangle}$
 ($\langle\sigma\rangle$ is the (cyclic) subgroup generated by $\sigma$) is equal to the number of
 cycles in the decomposition of the permutation $\sigma$ into disjoint cycles. This yields the statement.
\end{proof}

\begin{examples}
{\bf 1.} The proof of Proposition~\ref{prop:A_n} implies that a cyclic subgroup of $S_n$ satisfies PC if
and only if its generator (and therefore each element) is an even permutation. In particular, the subgroup
$A_3\subset S_3$ satisfies PC.

{\bf 2.} The subgroup $A_4\subset S_4$ does not satisfy PC since $\dim(\CC^4)^{A_4}=1$. This implies that,
for $n\ge 4$, the subgroup $A_n\subset S_n$ does not satisfy PC (since $A_n$ contains the subgroup $A_4$
permuting the first four coordinates).

{\bf 3.} The group $A_4$ contains the subgroup $S=\langle (12)(34), (13)(24) \rangle$ isomorphic to
$\ZZ_2 \times \ZZ_2$. The space $(\CC^4)^S$ is one-dimensional and therefore the subgroup $S$ does not satisfy PC.
(This was used in the example above.)

{\bf 4.} The group $\langle (12345), (12)(34) \rangle \subset S_5$ coincides with $A_5$ and therefore does not
satisfy PC. The group $\langle (12345), (14)(23) \rangle \subset A_5$ is isomorphic to the dihedral group
$D_{10}$ and satisfies PC.
\end{examples}

\section{Non-abelian Saito duality for invertible polynomials} \label{sect:main}

Let $f$, $\widetilde{f}$ and $S$ be as in Section~\ref{sect:Euler_Milnor} and also let
${\widehat{G}}_f=G_f\rtimes S$, ${\widehat{G}}_{\widetilde{f}}=G_{\widetilde{f}}\rtimes S$.
We shall first show that the equivariant Euler characteristic of the Milnor fibre $V_f$
with the $G_f \rtimes S$-action belongs to the subgroup $A^\rtimes(G_f\rtimes S) \subset A(G_f\rtimes S)$.

\begin{proposition} \label{prop:Burnside}
One has $\overline{\chi}^{\widehat{G}_f}(V_{f})\in A^\rtimes(G_f\rtimes S)$ $($and similarly 
$\overline{\chi}^{\widehat{G}_{\widetilde{f}}}(V_{\widetilde{f}}) \in
A^\rtimes(G_{\widetilde{f}} \rtimes S)${$)$}.
\end{proposition}

The proof of Proposition~\ref{prop:Burnside} will be postponed. Now we are ready to state
our main result.

\begin{theorem} \label{thm:main}
 If the subgroup $S\subset S_n$ satisfies PC, then one has
 \begin{equation}
   \overline{\chi}^{G_f \rtimes S}(V_{f})
  =(-1)^n D_{G_{\widetilde{f}}\rtimes S}^{\rtimes}\overline{\chi}^{G_{\widetilde{f}}\rtimes S}(V_{\widetilde{f}})\,. 
 \end{equation}
\end{theorem}

For the proofs of Proposition~\ref{prop:Burnside} and
Theorem~\ref{thm:main} we consider a decomposition of $\CC^n$ into certain 
unions of tori.

 For a subset $I\subset I_0=\{1,2,\ldots,n\}$, let
 $\CC^I :=\{(x_1, \ldots,x_n)\in\CC^n: x_i=0 {\text{ for }} i\notin I\}$,
 $(\CC^*)^I :=\{(x_1, \ldots,x_n)\in\CC^n: x_i\ne 0  {\text{ for }} i\in I, x_i=0 {\text{ for }} i\notin I\}$. One has
 $\CC^n=\bigsqcup_{I\subset I_0}(\CC^*)^I$. Each torus $(\CC^*)^I$ is invariant with respect to the actions of
 the groups $G_f$ and $G_{\widetilde{f}}$. Let $G_f^I\subset G_f$ and $G_{\widetilde{f}}^I\subset G_{\widetilde{f}}$
 be the isotropy subgroups of these actions on $(\CC^*)^I$. (The isotropy subgroups are the same for all points of
 $(\CC^*)^I$.)
 
 The group $S$ acts on the set $2^{I_0}$ of subsets of $I_0$. For a subset $I\subset I_0$, let
 $S^I :=\{\sigma\in S: \sigma I=I\}$ be the isotropy subgroup of $I$ for the action of $S$ on $2^{I_0}$.
 One has $S^{\overline{I}}=S^I$ where $\overline{I}:=I_0 \setminus I$.
 
 One has the decomposition
 $$
 \CC^n=\bigsqcup_{{\mathfrak{I}}\in 2^{I_0}/S} \left(\bigsqcup_{I\in \mathfrak{I}}(\CC^*)^I\right)\,,
 $$
 where the unions are over the orbits $\mathfrak{I}$ of the $S$-action on $2^{I_0}$, and over the
 elements $I$ of the orbit,
 and therefore
 $$
 V_f=\bigsqcup_{{\mathfrak{I}}\in 2^{I_0}/S} 
 \left(\bigsqcup_{I\in \mathfrak{I}}\left(V_f\cap(\CC^*)^I\right)\right)\,.
 $$
 The subset $\bigsqcup_{I\in \mathfrak{I}}\left(V_f\cap(\CC^*)^I\right)\subset V_f$ is $(G_f \rtimes S)$-invariant.
 Therefore
 \begin{equation} \label{eq:union}
 \chi^{G_f\rtimes S}(V_f)=\sum_{{\mathfrak{I}}\in 2^{I_0}/S} 
 \chi^{G_f\rtimes S}\left(\bigsqcup_{I\in \mathfrak{I}}\left(V_f\cap(\CC^*)^I\right)\right)\,.
 \end{equation}
 It is easy to see that
 \begin{equation} \label{eq:Ind}
 \chi^{G_f\rtimes S}\left(\bigsqcup_{I\in \mathfrak{I}}\left(V_f\cap(\CC^*)^I\right)\right)=
 \Ind_{G_f\rtimes S^I}^{G_f\rtimes S} \chi^{G_f\rtimes S^I}\left(V_f\cap(\CC^*)^I\right)\,.
 \end{equation}
 
 For $I\subset I_0$, let $f^I:=f_{\vert\CC^I}$; for $I\subset I_0$ and for a subgroup $T\subset S^I$,
 let $f^{I,T}:=f_{\vert(\CC^I)^T}$\,. Here the set $(\CC^I)^T$ of fixed points of $T$ in the coordinate subspace
 $\CC^I$ consists of the points whose coordinates are equal if and only if they can be sent one to the other by
 an element of $T$. If $f^I$ has an isolated critical point at the origin, then $f^{I,T}$ has an isolated critical
 point as well. For short we shall denote $f^{I_0,T}$ by $f^T$.
 
 The polynomial $f^I=f_{\vert(\CC^n)^I}$ has not more than $\vert I\vert$ monomials. It has $\vert I\vert$ monomials
 if and only if the polynomial $\widetilde{f}^{\,\overline{I}}$ has $\vert \overline{I}\vert$ monomials.
We treat the case that $f^I$ has less than $\vert I\vert$ monomials in the following lemma.

\begin{lemma} \label{lem:degenerate}
With the notations above, let $f^I$ have less than $\vert I\vert$ monomials. Then
$$
\chi^{G_f\rtimes S^I}\left(V_f\cap (\CC^*)^I\right)=0 \, .
$$
\end{lemma}

\begin{proof}
 We shall construct a free $\CC^*$-action on $(\CC^*)^I$ which preserves $f_{\vert(\CC^*)^I}$ and commutes with the
 $S$-action. The existence of such an action implies the statement of Lemma~\ref{lem:degenerate}.
 One can see that $f^I$ is the Sebastiani--Thom sum of some chains, some loops and some polynomials of the form
 \begin{equation}\label{eq:fragment}
  x_{i_1}^{p_1}x_{i_2}+x_{i_2}^{p_2}x_{i_3}+\ldots+x_{i_{m-1}}^{p_{m-1}}x_{i_m}
 \end{equation}
 (in different sets of variables).
 The group $S^I$ permutes isomorphic blocks and (possibly) rotates loops (and also permutes variables
 which do not participate in $f^I$). Moreover, either there are some variables (say, $x_{j_1}$, \dots, $x_{j_{\ell}}$,
 $\ell>0$) which do not participate in $f^I$ or there are some blocks of the form~(\ref{eq:fragment}). In the first
 case let $\lambda\in\CC^*$ act on $(\CC^*)^I$ by multiplying  all the variables $x_{j_1}$, \dots, $x_{j_{\ell}}$
 by $\lambda$ and keeping all the other variables fixed. In the second case one may have several blocks isomorphic
 to~(\ref{eq:fragment}) permuted by the group $S^I$. Let $\underline{q}=(q_1, \ldots, q_m)\in \ZZ^m\setminus\{0\}$
 be such that $q_{i_r}p_i+q_{i_{r+1}}=0$ for $1\le r\le m-1$. (Such $\underline{q}$ exists since it is defined by
 $(m-1)$ linear equations with respect to $m$ variables.) The action of $\CC^*$ on the variables $x_{i_1}$, \dots,
 $x_{i_m}$ defined by
 $$
 \lambda*(x_{i_1}, \ldots, x_{i_m})=(\lambda^{q_1}x_{i_1}, \ldots, \lambda^{q_m}x_{i_m})
 $$
 (and keeping all the other variables fixed) preserves $f^I$. If we define its action on the other $m$-tuples
 of variables obtained from these ones by the action of $S^I$ in the same way, we get a $\CC^*$-action which
 preserves $f^I$ and commutes with $S^I$.
 \end{proof}
 
 \begin{proof*}{Proof of Proposition~\ref{prop:Burnside}}
 According to Proposition~\ref{prop:exT} one has
 \begin{equation} \label{eq:torus}
  \chi^{G_f\rtimes S}(V_f\cap(\CC^*)^n)=\sum_{[T]\in{\rm Conjsub\,}S}\chi^{G_f\rtimes S}((V_f\cap(\CC^*)^n)^{([\{e\}\rtimes T])}) \in A^\rtimes(G_f\rtimes S)\,.
 \end{equation}
 Let $I \subset I_0$ be a proper subset. If $f^I$ has less than $\vert I \vert$ monomials then $\chi^{G_f\rtimes S^I}\left(V_f\cap (\CC^*)^I\right)=0$ according to Lemma~\ref{lem:degenerate}. Otherwise $f^I$ has $\vert I\vert$ monomials. In this case $f^I$ is an invertible polynomial in the variables $x_i$ with $i\in I$. Proposition~3 applied to $f^I$ implies
\begin{equation} \label{eq:torusI}
  \chi^{G_{f^I}\rtimes S^I}(V_{f^I}\cap (\CC^*)^I)=\sum_{[T]\in{\rm Conjsub\,}S^I}\chi^{G_{f^I}\rtimes S^I}((V_{f^I}\cap (\CC^*)^I)^{([\{e\}\rtimes T])})\,.
\end{equation}
The last sum lies in $A^\rtimes(G_{f^I}\rtimes S) \subset A^\rtimes(G_f\rtimes S)$. Therefore,
Equations (\ref{eq:union}), (\ref{eq:Ind}), (\ref{eq:torus}), and (\ref{eq:torusI}) imply the statement. 
 \end{proof*}
 
 For the proof of Theorem~\ref{thm:main} we need two lemmas.
 \begin{lemma} \label{lem:14} Under the imposed conditions one has
\begin{eqnarray}\label{eq:1)}
  \chi^{G_f\rtimes S}\left(V_f\cap(\CC^*)^n\right) & = & (-1)^{n-1}\left[G_f\rtimes S/\{ e \} \times S\right] \nonumber \\
 &  = & (-1)^{n-1}
 D^{\rtimes}_{G_{\widetilde{f}}\rtimes S}\left[G_{\widetilde{f}}\rtimes S/G_{\widetilde{f}}\rtimes S\right] \, .
 \end{eqnarray}
\end{lemma}

 \begin{lemma} \label{lem:16}
Under the imposed conditions one has, for the orbit $\mathfrak{I}$ of a non-empty proper subset $J \subset I_0$,
 \begin{equation}\label{eq:2')}
 \chi^{G_f\rtimes S}\left(\bigsqcup_{I\in{\mathfrak{I}}}\left(V_f\cap(\CC^*)^I\right)\right)=
 (-1)^{n} D^{\rtimes}_{G_{\widetilde{f}}\rtimes S}
 \chi^{G_{\widetilde{f}}\rtimes S}
 \left(\bigsqcup_{I\in{\mathfrak{I}}}\left(V_{\widetilde{f}}\cap(\CC^*)^{\overline{I}}\right)\right)\,.
 \end{equation}
 \end{lemma}

\begin{proof*}{Proof of Theorem~\ref{thm:main}}
Theorem~\ref{thm:main} follows from 
Lemma~\ref{lem:14} 
and Lemma~\ref{lem:16}  (together with (\ref{eq:1)}) with $f$ and $\widetilde{f}$
 interchanged).
 \end{proof*}
 
Lemma~\ref{lem:14} will follow from the following two lemmas.
\begin{lemma} \label{lem:19}
Under the imposed conditions one has
 \begin{equation}\label{eq:deepest}
 \chi^{G_f\rtimes S}\left(V_f\cap \left((\CC^*)^n\right)^{([\{e\}\rtimes S])}\right)=
 (-1)^{n-1}\left[G_f\rtimes S/\{e\}\rtimes S\right]\,.
 \end{equation}
\end{lemma}

\begin{lemma} \label{lem:20}
Under the imposed conditions one has, for a proper subgroup $T$ of $S$,
 \begin{equation}\label{eq:less_deep}
   \chi^{G_f\rtimes S}\left(V_f\cap \left((\CC^*)^n\right)^{([\{e\}\rtimes T])}\right)=0\,.
 \end{equation}
\end{lemma}

\begin{proof*}{Proof of Lemma~\ref{lem:19}}
For short, let $X:=V_f\cap(\CC^*)^n$. 
By Equation~(\ref{eq:torus}) it suffices to compute $\chi^{G_f\rtimes S}(X^{([\{e\}\rtimes T])})$ for a subgroup $T$ of $S$.
 For $T=S$ we have 
 $$
 X^{([\{e\}\rtimes S])}=\Ind_{N_{G_f\rtimes S}(\{e\}\rtimes S)}^{G_f\rtimes S}X^S\,,
 $$
 where $X^S=X^{(S)}$ is considered as an $N_{G_f\rtimes S}(\{e\}\rtimes S)$-space. 
 (Remember that the operation $\Ind$ was defined for spaces as well.) 
 The normalizer $N_{G_f\rtimes S}(\{e\}\rtimes S)$ is the semidirect (in fact direct) product $G_f^S\rtimes S$,
 where $G_f^S\subset G_f$ is the subgroup of $G_f$ consisting of the elements $g$ on which $S$ acts trivially,
 i.~e.\ $\varphi(\sigma)g=g$ for all $\sigma\in S$. This means that the subgroup $G_f^S$ consists of the
 elements $(\lambda_1, \ldots, \lambda_n)\in G_f$ such that $\lambda_i=\lambda_j$ for $i$ and $j$
 in the
 same $S$-orbit. Therefore $G_f^S$ coincides with the symmetry group $G_{f^S}$ of the function $f^S$
 which is again an invertible polynomial.
 Since $G_f^S$ acts freely on $X^S$ and $S$ acts trivially on it, one has
 $$
 \chi^{G_f^S\rtimes S}(X^S)=\chi\left(\left(V_{f^S}\cap((\CC^*)^n)^S\right)/G_f^S\right)\cdot
 \left[G_f^S\rtimes S/\{e\}\rtimes S\right]\,.
 $$
 Let $E^S$ be the exponent matrix of the invertible polynomial $f^S$.
 The formula of \cite[Theorem (7.1)]{Varch} implies that
\begin{equation} \label{eq:Varchenko}
 \chi\left(V_{f^S}\cap((\CC^*)^n)^S\right)=(-1)^{\dim (\CC^n)^S-1} \cdot \det E^S = (-1)^{\dim (\CC^n)^S-1} \cdot \vert G_f^S\vert.
\end{equation}
 But the PC property of $S$ implies
 $$
 (-1)^{\dim (\CC^n)^S-1}= (-1)^{n-1} .
 $$
 Therefore
 $$
 \chi\left(\left(V_{f^S}\cap((\CC^*)^n)^S\right)/G_f^S\right)=(-1)^{n-1},
 $$
 $$
 \chi^{G_f^S\rtimes S}\left((V_f\cap \left(\CC^*)^n\right)^{\{e\}\rtimes S}\right)=
 (-1)^{n-1}\left[G_f^S\rtimes S/\{e\}\rtimes S\right] \,.
 $$
 This implies the statement.
 \end{proof*}
 
 
 \begin{proof*}{Proof of Lemma~\ref{lem:20}}
 We keep the notation $X:=V_f\cap(\CC^*)^n$.
 We have
 $$
 \chi^{G_f\rtimes S}\left(X^{([\{e\}\times T])}\right)=
 \Ind_{N_{G_f\rtimes S}(\{e\}\times T)}^{G_f\rtimes S}
 \chi^{N_{G_f\rtimes S}(\{e\}\times T)}\left(X^{(T)}\right)\,.
 $$
 The normalizer of $\{e\}\times T$ in $G_f\rtimes S$ coincides with $G_{f^T}\rtimes N_S(T)$.
 The group $G_{f^T}\rtimes N_S(T)/\{e\}\times T$ acts freely on $X^{(T)}$. Therefore
 \begin{eqnarray*}
 \chi^{G_{f^T}\rtimes N_S(T)}\left(X^{(T)}\right)&=&
 \chi\left(X^{(T)}/\left(G_{f^T}\rtimes N_S(T)/\{e\}\times T\right)\right)\cdot\left[G_{f^T\rtimes N_S(T)}/\{e\}\times T\right]\\
 &=& \frac{\vert T\vert\cdot\chi\left(X^{(T)}\right)}{\vert G_{f^T}\rtimes N_S(T)\vert} \, \cdot \left[G_{f^T\rtimes N_S(T)}/\{e\}\times T\right].
\end{eqnarray*}
 
 Let us show that
 \begin{equation}\label{eq:nuli}
 \chi\left(X^{(T)}\right)=0\,.  
 \end{equation}
Assume that for all subgroups $U$ such that $T\varsubsetneq U\varsubsetneq S$ one has $\chi\left(X^{(U)}\right)=0$.
(In particular, this holds if $T$ is a maximal proper subgroup of $S$.) The equation $\chi\left(X^{(U)}\right)=0$
implies that, for $g\in G_{f^T}$, $\chi\left(X^{(g^{-1}Ug)}\right)=0$ (since $X^{(g^{-1}Ug)}=gX^{(U)}$).
Moreover, for $g_1,\,g_2\in G_{f^T}$, the spaces $X^{(g_1^{-1}Ug_1)}$  and $X^{(g_2^{-1}Ug_2)}$ either coincide
(if $g_1^{-1}Ug_1=g_2^{-1}Ug_2$) or do not intersect.

One has
\begin{equation}\label{eq:decopm}
 X^T\setminus \bigcup_{g\in G_{f^T}}X^{g^{-1}Sg}=X^{(T)}\cup\bigcup_{U:T\varsubsetneq U\varsubsetneq S}
 \bigcup_{g\in G_{f^T}} X^{(g^{-1}Ug)}.
\end{equation}
Due to the assumption one has
$$
\chi\left(\bigcup_{U:T\varsubsetneq U\varsubsetneq S}\bigcup_{g\in G_{f^T}} X^{(g^{-1}Ug)}\right)=0\,.
$$
Therefore $\chi(X^{(T)})=0$ if and only if $\chi\left(X^T\setminus \bigcup_{g\in G_{f^T}}X^{g^{-1}Sg}\right)=0$.
For $g_1,g_2\in G_{f^T}$, $g_1^{-1}S g_1=g_2^{-1}S g_2$ if and only if $g_1g_2^{-1}\in G_{f^S}$. This implies that
$$
\chi\left(X^T\setminus \bigcup_{g\in G_{f^T}}X^{g^{-1}Sg}\right)=
\chi(X^T)-\frac{\vert G_{f^T}\vert}{\vert G_{f^S}\vert}\chi\left(X^S\right)\,.
$$
 The same formula of \cite[Theorem (7.1)]{Varch} gives
 $$
 \chi(X^T)=(-1)^{\dim (\CC)^T-1}\vert G_{f^T}\vert\,,\quad \chi(X^S)=(-1)^{\dim (\CC)^S-1}\vert G_{f^S}\vert\,.
 $$
 The condition PC gives that signs in these equations are
 equal to $(-1)^{n-1}$. Therefore
 $$
 \chi\left(X^T\setminus \bigcup_{g\in G_{f^T}}X^{g^{-1}Sg}\right)=
 (-1)^{n-1}\left(\vert G_{f^T}\vert-\frac{\vert G_{f^T}\vert}{\vert G_{f^S}\vert}\vert G_{f^S}\vert\right)=0\,.
 $$
 This proves the statement.
 \end{proof*}
 
 Now we shall proceed to the proof of Lemma~\ref{lem:16}. Let 
$f^I$ have $\vert I\vert$ monomials. In this case $f^I$ is an invertible polynomial in the variables $x_i$
 with $i\in I$. If the action of $S^I$ on $\CC^I$ satisfies PC (in this case the action of $S^{\overline{I}}=S^I$
 on $\CC^{\overline{I}}$ satisfies PC as well), one can explicitly compute both sides  of~(\ref{eq:2')})
 (and in this
 way prove it) repeating the arguments  of the proof of Lemma~\ref{lem:14} for $I=I_0$. However, in general, the action of $S^I$ on $\CC^I$ does not
 satisfy PC. Nevertheless for any subgroup $T\subset S^I$, one has
 \begin{equation}\label{eq:quasi-PC}
 \dim(\CC^I)^T-\dim(\CC^I)^{S^I}\equiv \dim(\CC^{\overline{I}})^{T}-\dim(\CC^{\overline{I}})^{S^I}\mod 2\,.
 \end{equation}
 
We shall consider the poset of subgroups $T$ of $S^I$. It is represented by a graph which is called the {\em Hasse diagram}. We colour this graph by the parities of the codimensions of $(\CC^I)^T$ in $\CC^I$. We call this graph the {\em coloured Hasse diagram} of $S^I$ and denote it by $\Gamma^I$.
 Lemma~\ref{lem:16} will then follow from the following lemma.
 
\begin{lemma} \label{lem:25}
Under the imposed conditions one has, for a non-empty proper subset $I \subset I_0$ such that $f^I$ has $\vert I \vert$ monomials,
  \begin{equation}\label{eq:independence}
  \chi^{G_{f^I}\rtimes S^I}\left(V_{f^I}\cap (\CC^*)^I\right)=
  \sum_{[T]\in{\rm Conjsub\,}S^I}a_{[T]}\left[G_{f^I}\rtimes S^I/\{e\}\rtimes T\right]
 \end{equation}
 for some coefficients $a_{[T]}$, which depend only on the coloured Hasse diagram $\Gamma^I$.
\end{lemma}
 
\begin{proof*}{Proof of Lemma~\ref{lem:16}}
Consider the function $f^I$. If $f^I$ has less than $\vert I \vert$ monomials, then Lemma~\ref{lem:degenerate} implies that $\chi^{G_f\rtimes S^I}\left(V_f\cap (\CC^*)^I\right)=0$. Otherwise,
Lemma~\ref{lem:25} implies the following equation for the orbit $\mathfrak{I}$ of the subset $I$
\begin{equation} \label{eq:left17}
  \chi^{G_{f}\rtimes S}\left(\bigsqcup_{J \in \mathfrak{I}} V_f\cap (\CC^*)^J\right)=
  \sum_{[T]\in{\rm Conjsub\,}S^I}a_{[T]}\left[G_{f}\rtimes S/G_f^I\rtimes T\right]\,.
\end{equation}
 By Lemma~\ref{lem:25}, the coefficients $a_{[T]}$ depend only on the coloured Hasse diagram $\Gamma^I$. But $\Gamma^I= \Gamma^{\overline{I}}$, since the poset of subgroups $T$ of $S^I$ is the same for
 $S^{\overline{I}}=S^I$ and the parities of the codimensions of $(\CC^I)^T$ in $\CC^I$ and of
 $(\CC^{\overline{I}})^T$ in $\CC^{\overline{I}}$ are also the same due to the condition PC: see
 Equation~(\ref{eq:quasi-PC}). This together with the fact that $G_{\widetilde{f}}^{\overline{I}}=\widetilde{G^I_f}$ (see~\cite[Lemma 1]{EG-BLMS})
 implies~(\ref{eq:2')}).
 \end{proof*}
 
 \begin{proof*}{Proof of Lemma~\ref{lem:25}}
 By Equation~(\ref{eq:torusI}), it suffices to compute 
$$
\chi^{G_{f^I}\rtimes S^I}\left(V_{f^I}\cap \left((\CC^*)^I\right)^{([\{e\}\rtimes T])}\right)
$$
for a subgroup $T \subset S$. But this
 is just the summand $a_{[T]}\left[G_{f^I}\rtimes S^I/\{e\}\rtimes T\right]$ in (\ref{eq:independence}). For $T=S^I$ this summand
 is computed in the same way as the corresponding summand for $I=I_0$ above and thus is equal to
 $$
 (-1)^{\dim(\CC^I)^{S^I}-1}\left[G_{f^I}\rtimes S^I/\{e\}\rtimes S^I\right]\,.
 $$
 (Pay attention that we do not substitute $(-1)^{\dim(\CC^I)^{S^I}-1}$ by $(-1)^{\dim(\CC^I)-1}$ because, in general,
 the action of $S^I$ on $\CC^I$ does not satisfy PC.) We shall show that, for a proper subgroup $T$ of $S^I$, the
 Euler characteristic $\chi\left( V_{f^I}\cap\left(\CC^*)^I\right)^{([\{e\}\rtimes T])}\right)$ is a multiple of
 $\vert G_{f^{I,T}}\vert$ with the factor depending only on the coloured Hasse diagram $\Gamma^I$.
 
 For short, let $X_I :=V_{f^I}\cap(\CC^*)^I$.
 Assume that for all subgroups $U$ such that $T\varsubsetneq U\varsubsetneq S$ the Euler characteristic
 $\chi\left(X_I^{(U)}\right)$ is a multiple of $\vert G_{f^{I,U}}\vert$. The union
 $\bigcup_{g\in{G_{f^{I,T}}}}X_I^{(g^{-1} Ug)}$ contains $\vert G_{f^{I,T}}\vert/\vert G_{f^{I,U}}\vert$
 different homeomorphic spaces. Therefore the Euler characteristic of it is a multiple of $\vert G_{f^{I,T}}\vert$.
 Using an analogue of Equation~(\ref{eq:decopm}), it is sufficient to show that 
 $\chi\left(X_I^T\setminus \bigcup_{g\in G_{f^{I,T}}}X_I^{g^{-1}Sg}\right)$
 is a multiple of $\vert G_{f^{I,T}}\vert$ with the factor depending only on the coloured Hasse diagram $\Gamma^I$.
 As above we have
$$
\chi\left(X_I^T\setminus \bigcup_{g\in G_{f^{I,T}}}X_I^{g^{-1}S^Ig}\right)=
\chi(X_I^T)-\frac{\vert G_{f^{I,T}}\vert}{\vert G_{f^{I,S^I}}\vert}\chi\left(X_I^{S^I}\right)\,,
$$
 $\chi(X_I^T)=(-1)^{\dim (\CC^I)^T-1}\vert G_{f^{I,T}}\vert$,
 $\chi(X_I^{S^I})=(-1)^{\dim (\CC^I)^{S^I}-1}\vert G_{f^{I,S^I}}\vert$. 
 Therefore
 $$
 \chi\left(X_I^T\setminus \bigcup_{g\in G_{f^{I,T}}}X_I^{g^{-1}S^Ig}\right)=
 \left((-1)^{\dim (\CC^I)^T-1}-(-1)^{\dim (\CC^I)^{S^I}-1}\right)\cdot\vert G_{f^{I,T}}\vert\,,
 $$
 i.~e.\ it is a multiple of $\vert G_{f^{I,T}}\vert$ with the factor (equal either to $0$ or to $\pm2$)
 only depending on the coloured Hasse diagram $\Gamma^I$. Now
 $$
 a_{[T]}=\chi\left(X_I^{(T)}/G_{f^{I,T}}\rtimes N_{S^I}(T)\right)=
 \frac{\vert T\vert}{\vert N_{S^I}(T)\vert}\cdot\frac{\chi\left(X_I^{(T)}\right)}{\vert G_{f^{I,T}}\vert}
 $$
 which shows that the coefficients $a_{[T]}$ only depend on the coloured Hasse diagram $\Gamma^I$.
\end{proof*}

\section{Examples of BHHT dual pairs with and without PC} \label{sect:dual-non-dual}
\begin{table}
\begin{tabular}{|r|c|p{54mm}|p{47mm}|p{19mm}|}      \hline
N& $f$ & \centering{$G$} & \centering{$S$} & $(h^{1,1},h^{2,1})$ \\
\hline
\hline
$\begin{array}{r} 2 \\ 83 \end{array}$ & $X_1$    & $\begin{array}{l} J \\ \left\{\frac{1}{5}(k_1,\ldots, k_5): 5\vert\sum k_i\right\} \end{array}$           & $\ZZ_2: (12)(34)$ & $\begin{array}{c} (3,59)\\ (59,3) \end{array}$ \\
\hline
$\begin{array}{r} 3 \\ 84 \end{array}$ & $X_1$    & $\begin{array}{l} J \\ \left\{\frac{1}{5}(k_1,\ldots, k_5): 5\vert\sum k_i\right\} \end{array}$           & $\ZZ_3: (123)$ & $\begin{array}{c} (5,49)\\ (49,5) \end{array}$ \\
\hline
$\begin{array}{r} 20 \\ 57 \end{array}$ & $X_1$    & $\begin{array}{l} \frac{1}{5}(1,4,1,4,0), J \\ \frac{1}{5}(0,0,1,1,3), \frac{1}{5}(1,4,4,1,0), J  \end{array}$        & $\ZZ_2: (12)(34)$ & $\begin{array}{c} (13,17) \\ (17,13) \end{array}$ \\
\hline
$\begin{array}{r} 21 \\58 \end{array}$& $X_1$    & $\begin{array}{l} \frac{1}{5}(1,4,0,0,0), J \\  \frac{1}{5}(1,1,0,0,3), \frac{1}{5}(0,0,1,4,0), J  \end{array}$        & $\ZZ_2: (12)(34)$ & $\begin{array}{c} (5,33) \\ (33,5) \end{array}$ \\
\hline
$\begin{array}{r} 22 \\59 \end{array}$& $X_1$    & $\begin{array}{l} \frac{1}{5}(1,4,2,3,0), J \\  \frac{1}{5}(0,0,1,1,3), \frac{1}{5}(1,4,2,3,0), J  \end{array}$        & $\ZZ_2: (12)(34)$ & $\begin{array}{c} (3,19) \\ (19,3) \end{array}$ \\
\hline
$\begin{array}{r} 42 \\ 73 \end{array}$& $X_1$    & $\begin{array}{l}  \frac{1}{5}(0,1,2,3,4), J \\   \frac{1}{5}(0,1,4,4,1), \frac{1}{5}(0,1,2,3,4), J \end{array}$        & $\ZZ_5: (12345)$ & $\begin{array}{c} (1,5) \\ (5,1) \end{array}$ \\
\hline
$\begin{array}{r} 11 \\43 \end{array}$ & $X_{14}$    & $\begin{array}{l} \frac{1}{3}(1,2,0,0,0),J \\ \frac{1}{15}(3,3,5,10,9), J   \end{array}$        & $\ZZ_2: (12)(34)$ & $\begin{array}{c}  (5,33) \\ (33,5) \end{array}$ \\
\hline
$\begin{array}{r}12 \\ 44 \end{array}$ & $X_{14}$    & $\begin{array}{l} \frac{1}{3}(1,2,1,2,0), J \\  \frac{1}{15}(13,8,2,7,0), J  \end{array}$        & $\ZZ_2: (12)(34)$ & $\begin{array}{c} (5,25) \\ (25,5) \end{array}$ \\
\hline
$\begin{array}{r}13 \\ 47 \end{array}$ & $X_{14}$    & $\begin{array}{l} \frac{1}{3}(1,2,1,2,0), J \\  \frac{1}{15}(13,8,2,7,0), J  \end{array}$       & $\ZZ_2: (13)(24)$ & $\begin{array}{c} (11,19) \\ (19,11) \end{array}$ \\
\hline
\hline
$\begin{array}{r} 80 \\ * \end{array}$ & $X_{15}$    & $\begin{array}{l} \frac{1}{41}(1,-4,16,18,10), J \\  J  \end{array}$       & $\ZZ_5: (12345)$ & $\begin{array}{c} (21,1) \\ (1,21) \end{array}$ \\
\hline
$\begin{array}{r} 82 \\ * \end{array}$ & $X_1$    & $\begin{array}{l} \frac{1}{5}(0,1,4,4,1), \frac{1}{5}(0,1,2,3,4), J \\  \frac{1}{5}(3,1,4,2,0), J  \end{array}$       & $D_{10}: (12345), (25)(34)$ & $\begin{array}{c} (11,3) \\ (3,11) \end{array}$ \\
\hline
\hline
$\begin{array}{r} 7 \\86 \end{array}$ & $X_1$    & $\begin{array}{l} J \\   \left\{\frac{1}{5}(k_1,\ldots, k_5): 5\vert\sum k_i\right\}  \end{array}$       & $\ZZ_2\times \ZZ_2: (12)(34), (13)(24)$ & $\begin{array}{c} (7,35) \\ (41,1) \end{array}$ \\
\hline
$\begin{array}{r} 25 \\90 \end{array}$ & $X_1$    & $\begin{array}{l} J \\  \left\{\frac{1}{5}(k_1,\ldots, k_5): 5\vert\sum k_i\right\}  \end{array}$       & $A_4: (123), (12)(34)$ & $\begin{array}{c} (7,27) \\ (29,5) \end{array}$ \\
\hline
$\begin{array}{r} 62 \\ 91 \end{array}$ & $X_1$    & $\begin{array}{l} J \\  \left\{\frac{1}{5}(k_1,\ldots, k_5): 5\vert\sum k_i\right\}  \end{array}$        & $A_5: (12345), (123)$ & $\begin{array}{c} (5,13) \\ (15,3) \end{array}$ \\
\hline
$\begin{array}{r} 26 \\ 63 \end{array}$ & $X_{14}$    & $\begin{array}{l} \frac{1}{3}(1,2,1,2,0), J \\  \frac{1}{15}(13,8,2,7,0), J  \end{array}$       & $\ZZ_2\times \ZZ_2: (12)(34), (13)(24)$ & $\begin{array}{c} (11,15) \\ (21,5) \end{array}$ \\
\hline
\end{tabular}
\caption{BHHT dual pairs}\label{tab:table}
\end{table}

In \cite{Yu}, one has a list of some Calabi--Yau threefolds defined in $\CC\PP^4$ by homogeneous polynomials
of degree 5 with actions of certain (in general non-abelian) finite groups with the pairs of orbifold
Hodge--Deligne numbers $(h^{1,1},h^{2,1})$ (and the orbifold Euler characteristics equal to $2(h^{1,1}-h^{2,1})$).
These pairs are equal to the Hodge numbers of the crepant resolutions of the corresponding threefolds.
The majority of these threefolds are defined by invertible polynomials $f$ and by groups of the form $G\rtimes S$,
where $G\subset G_f$, $S\subset S_n$.

In \cite{Mukai}, the corresponding computations were made for the same polynomials and the liftings of the
corresponding groups to the so-called Landau--Ginzburg orbifolds, i.~e.\ to the quantum cohomology groups in terms of
\cite{FJR2015}. The corresponding Hodge--Deligne numbers $(h^{1,1},h^{2,1})$ appeared to be symmetric to those for
the corresponding Calabi--Yau threefolds in \cite{Yu}. The table from \cite{Yu} is reproduced in \cite{Mukai}:
Table~1 therein. We shall use it for reference since it is more convenient: in 
this table the examples are numbered. 
In the course of a thorough analysis of the table we have found that (up to possible renumbering
of the variables) it
contains 13 pairs of BHHT dual invertible polynomials with non-abelian symmetry groups. They are listed below
in Table~1. (The pairs $21 \leftrightarrow 58$ and $11 \leftrightarrow 43$ were detected by Takahashi.)

We use the following notations (partially taken from \cite{Yu}). The numbers of the
examples (the first column) correspond to their numbering in \cite[Table 1]{Mukai}. The names of the polynomials $f$
refer to:
\begin{eqnarray*}
 X_1:&\ & x_1^5+x_2^5+x_3^5+x_4^5+x_5^5\,;\\
 X_{14}:&\ & x_1^4x_2+x_2^4x_1+x_3^4x_4+x_4^4x_3+x_5^5\,;\\
 X_{15} : & \ & x_1^4x_2+ x_2^4x_3 + x_3^4x_4 + x_4^4x_5 + x_5^4x_1\,. 
\end{eqnarray*}
These polynomials are self-dual, i.~e.\ $\widetilde{f} =  f$ (up to a permutation of the variables in the latter
case). Elements of the group $G_f$ (generators of the subgroups $G$ in the Table) are denoted by
$\frac{1}{m}(a_1, \ldots, a_5)$ with integers $a_i$ which is a short notation for the operator
$$
\text{diag\,}\left(\exp\left(\frac{2\pi a_1i}{m}\right),\ldots,\left(\frac{2\pi a_5i}{m}\right)\right)\,.
$$
The element $J$ is the exponential grading operator $J=\frac{1}{5}(1,1,1,1,1)$ (represented by the monodromy
transformation of the polynomial $f$).

The first example in each pair was copied from~\cite{Yu} (or from~\cite{Mukai}). The descriptions of the dual ones
differ from those in~\cite{Yu} by permutations of the coordinates (this was made to have BHHT dual groups) and by
presentations of the corresponding groups $G$. In particular, in Examples~44 and 47 we give other generators
of the groups to make explicit that the groups are the same as in Example~63. The description of Example~73
is adapted to our notations: we have listed generators of $G$ explicitly.

In this table, the first 9 pairs satisfy PC and the last 4 pairs do not satisfy it: see Section~\ref{sect:PC}.
There are added two more pairs: Examples~80 and 82 (satisfying PC) plus their BHHT duals which are not in
\cite[Table~5]{Yu}.
The Hodge numbers are easily computed in the way used in \cite{Yu}. (In the case dual to Example~80,
the corresponding Calabi--Yau threefold is the quotient of the smooth Klein quintic by a free action of the
cyclic group $\ZZ_5$.)
One can see that in all the cases satisfying PC the Hodge numbers of the dual pairs are symmetric to each other,
whereas in all the other cases (not satisfying PC) they are not. 
Moreover, almost all pairs of Hodge numbers indicated in \cite[Theorem~3.20]{Yu} without mirror ones correspond
to the Calabi--Yau threefolds in \cite[Table~5]{Yu} defined either by non-invertible polynomials or by
semi-direct products of groups not satisfying PC. This indicates that the condition PC seems to be
necessary for the mirror symmetry of BHHT dual pairs.

\begin{remark}
The pair $(3,11)$ of Hodge numbers of the threefold BHHT dual to Example~82 shows that the list of pairs in
\cite[Theorem~3.20]{Yu} is not complete (in contradiction to \cite[Remark~3.21]{Yu}).
\end{remark}


\bigskip
\noindent Leibniz Universit\"{a}t Hannover, Institut f\"{u}r Algebraische Geometrie,\\
Postfach 6009, D-30060 Hannover, Germany \\
E-mail: ebeling@math.uni-hannover.de\\

\medskip
\noindent Moscow State University, Faculty of Mechanics and Mathematics,\\
Moscow, GSP-1, 119991, Russia\\
E-mail: sabir@mccme.ru
\end{document}